\documentclass[twoside,a4paper,reqno,11pt]{amsart}

\usepackage[top=28mm,right=30mm,bottom=28mm,left=30mm]{geometry}
\usepackage{amsfonts, amsmath, amssymb, mathrsfs, bm, latexsym, stmaryrd, array, mathtools, bold-extra, xcolor}
\usepackage[pagebackref]{hyperref}
\usepackage{caption}
\usepackage{listings}
\usepackage{arydshln}
\usepackage[all]{xy}
\usepackage{multirow}

\usepackage{enumitem}
\setlist[enumerate]{font={\rm},itemsep=0.2\baselineskip}
\setlist[enumerate,1]{label={(\roman*)}}
\setlist[enumerate,2]{label={(\arabic*)}}


\newtheorem{theorem}{Theorem}[section]

\newtheorem{corollary}[theorem]{Corollary}

\newtheorem{lemma}[theorem]{Lemma}

\theoremstyle{definition}
\newtheorem{definition}[theorem]{Definition}

\newtheorem*{conj*}{Conjecture}

\theoremstyle{remark}
\newtheorem{remark}[theorem]{Remark}

\def\GL{{\mathrm{GL}}}
\def\GammaL{{\Gamma \mathrm{L}}}

\def\SL{{\mathrm{SL}}}
\def\PSL{{\mathrm{PSL}}}

\def\Sp{{\mathrm{Sp}}}
\def\PSp{{\mathrm{PSp}}}
\def\SU{{\mathrm{SU}}}
\def\PSU{{\mathrm{PSU}}}

\def\G{{\mathrm{G}}}

\def\A{{\mathrm{A}}}

\def\F{{\mathrm{F}}}
\def\E{{\mathrm{E}}}

\def\L{{\mathrm{L}}}
\def\K{{\mathrm{K}}}

\def\bbZ{{\mathbb{Z}}}
\def\bbF{{\mathbb{F}}}

\def\Tr{{\mathrm{Tr}}}

\def\bfZ{{\mathbf{Z}}}

\def\Aut{{\mathrm{Aut}}}

\def\End{{\mathrm{End}}}
\def\Gal{{\mathrm{Gal}}}

\def\leqs{\leqslant}
\def\geqs{\geqslant}

\def\le{\leqslant}

\def\l{\langle}
\def\r{\rangle}

\def\Magma{{\sc Magma}}
\def\Sz{{\mathrm{Sz}}}

\begin{document}

\title{A proof of {Gross}' conjecture on $2$-automorphic $2$-groups}
\author{Cai Heng Li}

\address{SUSTech International Center for Mathematics, and Department of Mathematics, Southern University of Science and Technology, Shenzhen, Guangdong, China}
\email{lich@sustech.edu.cn {\text{\rm(Li)}}}

\author{Yan Zhou Zhu}

\address{SUSTech International Center for Mathematics, and Department of Mathematics, Southern University of Science and Technology, Shenzhen, Guangdong, China}
\email{zhuyz@mail.sustech.edu.cn {\text{\rm(Zhu)}}}

\begin{abstract}
	Building upon previous results, a classification is given of finite $p$-groups whose elements of order $p$ are all fused.
	This particularly confirms a conjecture of Gross proposed in 1976 on $2$-automorphic $2$-groups, which are $2$-groups with involutions forming a single fusion class.
	As a consequence, two open problems regarding AT-groups and FIF-groups are solved.
\end{abstract}

\maketitle

\section{Introduction}

It is well-known that finite $p$-groups with a single subgroup of order $p$ are cyclic or a generalized quaternion group~\cite[page 199]{gorenstein1980Finite}.
The class of finite $p$-groups with a single fusion class of elements of order $p$ has been extensively studied.
(Recall that a {\it fusion class} of a group $N$ is an orbit of $\Aut(N)$ on the elements.)
The study was initiated by G.\,Higman~\cite{higman1963Suzuki} in 1963 when he classified finite non-abelian $2$-groups (with more than one involution) that have cyclic groups of automorphisms transitive on the involutions, called {\it Suzuki $2$-groups}.

For odd primes $p$, Shult~\cite{shult1969finite} showed that homocyclic $p$-groups are the only $p$-groups of which elements of order $p$ form a single fusion class.
Shaw~\cite{shaw1970Sylow} showed that a finite non-abelian $2$-group (with more than one involution) which has a solvable group of automorphisms transitive on the involutions is a Suzuki $2$-group.
Higman's classification associated with Shaw's work forms Sections~5-7 of Chapter VII of the book written by Huppert and Blackburn~\cite{huppert1982Finite}.

In 1976, Gross~\cite{gross1976automorphic} studied finite $2$-groups with a single fusion class of involutions, called \textit{$2$-automorphic $2$-groups}.
He further proposed a conjecture that (roughly speaking) $2$-automorphic $2$-groups are Suzuki $2$-groups.
This conjecture has been partially solved by Bryukhanova~\cite{bryukhanova1981Automorphism} and Wilkens~\cite{wilkens1996note}, and is recorded as Problem 7.31 in the Kourovka notebook~\cite{mazurov1995Unsolved}.

Building upon previous results, the main theorem of this paper obtains a classification of finite $p$-groups of which elements of order $p$ form a single fusion class, completing the long-term classification problem and confirming Gross' conjecture.

\begin{theorem}\label{thm:gross}
	Let $N$ be a finite $p$-group such that $\Aut(N)$ is transitive on the set of all elements of order $p$.
	Then one of the following statements holds:
	\begin{enumerate}
		\item $N$ is a homocyclic $p$-group;
		\item $N$ is a generalized quaternion $2$-group;
		\item $N$ is a Suzuki $2$-group, and $\Aut(N)$ is solvable with a cyclic subgroup acting transitively on the set of involutions.
	\end{enumerate}
\end{theorem}

We use $G^{(\infty)}$ for the unique perfect group appearing in the derived series of $G$.
The proof of Theorem~\ref{thm:gross} for $2$-groups depends on certain representations of finite transitive linear groups over $\bbF_2$, stated below, which is independently interesting.

\begin{theorem}\label{thm:dual}
	Let $V=\bbF_2^n$, $G\leqs\GL(V)=\GL_n(2)$ be non-solvable, and $M=\Lambda^2_{\bbF_2}(V)/W$ for some $\bbF_2 G$-submodule $W$ of $\Lambda^2_{\bbF_2}(V)$.
	Assume that $\dim M=\dim V$, the natural action of $G$ on $V$ is transitive on the non-zero vectors of $V$, and the induced action of $G$ on $M$ is transitive on the non-zero vectors of $M$.
	Then $n=3m$, $G^{(\infty)}=\SL_3(2^m)<\GL_{3m}(2)$, and $M$ is the dual $\bbF_2 G^{(\infty)}$-module of $V$.
\end{theorem}

Solving Gross' conjecture is also an important part for the study of some other classes of groups.
A group $N$ is called an \textit{AT-group} if $\Aut(N)$ is transitive on the elements of the same order.
Zhang~\cite{zhang1992finite} presents a description of finite AT-groups.
However, the problem of determining AT-$2$-groups has been left open, as commented that ``If Gross' conjecture on $2$-automorphic $2$-groups is true, ...... We do not know if Suzuki $2$-groups are AT-groups" before Lemma 2.2 of~\cite{zhang1992finite}.
A group is called an {\it FIF-group} if any two elements of the same order are fused or inverse-fused.
Praeger and the first-named author~\cite{li1997Finitea} give a description of finite FIF-groups, with the problem of classifying FIF-$2$-groups left open, see~\cite[Problem 1.5]{li1997Finitea}.
As a corollary of Theorem~\ref{thm:gross}, these problems are solved.
\begin{corollary}\label{cor:AT-2-gps}
	Let $N$ be a non-abelian finite $2$-group with more than one involution.
	Then the following statements are equivalent:
	\begin{enumerate}
		\item cyclic subgroups of the same order are conjugate under $\Aut(N)$;
		\item any two elements of $N$ of the same order are conjugate or inverse-conjugate under $\Aut(N)$;
		\item elements of $N$ of the same order are conjugate under $\Aut(N)$;
		\item $N=A_2(n,\theta)$, $B_2(n)$, or $P(\epsilon)$; see Definitions~$\ref{def:apn}$, $\ref{def:bpn}$ and $\ref{def:pepsilon}$.
	\end{enumerate}
\end{corollary}

The three families of groups in Corollary~\ref{cor:AT-2-gps}(iv) are exactly all $2$-groups with $3$ fusion classes, which were also obtained by Bors and Glasby in~\cite{bors2020Finitea}.
We remark that the automorphism groups of the groups $A_2(n,\theta)$, $B_2(n)$ and $P(\epsilon)$ are determined in Lemmas~\ref{lem:antheta}, \ref{lem:bpn} and \ref{lem:epsilon}, respectively.

The layout of the paper is as follows. 
In Section 2, we introduce some methods in representation theory and prove Theorem~\ref{thm:dual}.
We prove Theorem~\ref{thm:gross} in Section~3 by linking special $p$-groups with representation theory and using Theorem~\ref{thm:dual}.
The final section contains the proof of Corollary~\ref{cor:AT-2-gps} together with the constructions of groups in the Corollary.

\subsection*{Acknowledgments}
The authors acknowledge the support of NNSFC grant no.~11931005.
Moreover, they would like to thank the anonymous referee for their careful reading of the paper and valuable suggestions.

\section{Transitive linear groups and their exterior squares}

All fields and groups are finite in this section.
For a field $\F$ and a group $G$, an \textit{$\F G$-reprensentation} is a group homomorphism $\rho$ from $G$ to $\GL(V)$ for some $\F$-space $V$.
Every $\F G$-reprensentation $\rho$ induces an $\F G$-module structure on $V$ by $(xg)\cdot v=x\cdot v^{\rho(g)}$ for $x\in \F$ and $g\in G$.
If $\E$ is an extension field of $\F$, then $V^\E:=V\otimes \E$ is naturally an $\E G$-module.
Some elementary properties of the functor $- \otimes \E$ can be found in~\cite[Section VII.1]{huppert1982Finite}, \cite[Chapter 9]{aschbacher2000Finite} and \cite[Chapter 9]{isaacs1976Character}.
For an $\F G$-module $V$ with representation $\rho$, we use the following two notations:
\begin{itemize}
	\item denote by $V^\sigma$ the $\F G$-module induced by $\sigma\circ \rho$ for any field automorphism $\sigma$ of $\F$;
	\item denote by $V_\K$ the $\K G$-module of $V$ for subfield $\K$ of $\F$.
\end{itemize}

Some important propositions are recorded blow.

\begin{lemma}\label{lem:prop}
	Suppose that $\K\subset \F\subset \E$ are field extensions, and $V, W$ are $\F G$-modules.
	Then the following statements hold:
	\begin{enumerate}
		\item {\rm\cite[Chapter VII, 1.16(a)]{huppert1982Finite}} $(V_\K)^\F\cong \sum_{\sigma\in\Gal(\F/\K)}V^\sigma$;
		\item {\rm\cite[25.7(2)]{aschbacher2000Finite}} if there exists an $\E G$-submodule $U$ of $V^\E$ such that $U^\sigma=U$ for any $\sigma\in\Gal(\E/\F)$, then $U=(V_0)^\E$ for some $\F G$-submodule $V_0$ of $V$;
		\item {\rm[Deuring-Noether, see~\cite[Chapter VII, 1.22]{huppert1982Finite}]} $V^\E\cong W^\E$ if and only if $V\cong W$.
	\end{enumerate}
\end{lemma}

The \textit{exterior square} $\Lambda^2_{\F}(V)$ of an $\F G$-module $V$ is $(V\otimes V)/S$ with $S=\langle v\otimes v:v\in V\rangle$.
Note that if $V=V_1\oplus\cdots \oplus V_n$, we have
\[\Lambda^2_{\F}(V)\cong \bigoplus_{i=1}^n\Lambda_{\F}^2(V_i)\oplus \bigoplus_{1\leqs i<j\leqs n} V_i\otimes V_j.\]
Then we can prove the following lemma for decomposing exterior squares.

\begin{lemma}\label{lem:decompoext}
	Suppose that $V$ is an $\F G$-module and $\E$ is an extension field of $\F$.
	If $U$ is an $\E G$-module such that $U_{\F}\cong V$, then the following statements hold:
	\begin{enumerate}
		\item $V^\E\cong\bigoplus_{i=1}^f U^{\phi^i}$ where $\phi$ is a generator of $\Gal(\E/\F)$ and $f=[\E:\F]$;
		\item $\Lambda^2_\F(V)=A\oplus \bigoplus_{j=1}^{\lfloor \frac{f}{2}\rfloor}B_j$ where
		\begin{enumerate}
			\item $A\cong \Lambda_\E^2(U)_\F$;
			\item $B_j\cong \bigl(U\otimes U^{\phi^j}\bigr)_\F$, when $1\leqs j < f/2$;
			\item $2\cdot B_{f/2}\cong \bigl(U\otimes U^{\phi^{f/2}}\bigr)_\F$, when $f$ is even.
		\end{enumerate}
	\end{enumerate}
\end{lemma}
\begin{proof}
	Part~(i) is directly deduced by Lemma~\ref{lem:prop}(i).

	Note that $\Lambda^2_\F(V)=(V\otimes V)/S$ with $S=\langle v\otimes v:v\in V\rangle$, then $\Lambda^2_\F(V)^\E\cong (V\otimes V)^\E/S^\E\cong \Lambda_\E^2(V^\E)$.
	Thus, we obtain that
	\[\begin{aligned}
		\Lambda^2_\F(V)^\E&\cong \Lambda_\E^2(V^\E)\cong
		\Lambda_\E^2(U\oplus U^\phi\oplus\cdots\oplus U^{\phi^{f-1}})\\
		&\cong \bigoplus_{i=1}^f\Lambda_\E^2(U)^{\phi^i}\oplus \bigoplus_{1\leqs i<j\leqs f}U^{\phi^i}\otimes U^{\phi^j}.
	\end{aligned}\]
	We use the notations:
	\begin{itemize}
		\item $\widetilde{A}\cong \bigoplus_{i=1}^f\Lambda_\E^2(U)^{\phi^i}$;
		\item $\widetilde{B}_j\cong \bigoplus_{i=1}^{f}\bigl(U\otimes U^{\phi^j}\bigr)^{\phi^i}$ for $1\leqs j<\frac{f}{2}$; and
		\item $\widetilde{B}_{f/2}=\bigoplus_{i=1}^{f/2}\bigl(U\otimes U^{\phi^{f/2}}\bigr)^{\phi^i}$ when $f$ is even.
	\end{itemize}
	Hence, $\Lambda^2_\F(V)^\E\cong \widetilde{A}\oplus \bigoplus_{j=1}^{\lfloor \frac{f}{2}\rfloor}\widetilde{B}_j$ by rearranging the summands.
	Note that $\phi$ fixes $\widetilde{A}$ and $\widetilde{B}_j$ for each $1\leqs j\leqs\frac{f}{2}$, then there exist $\F G$-submodules $A$ and $B_j$'s of $\Lambda^2_\F(V)$ such that $A^\E\cong \widetilde{A}$ and $B_j^\E\cong\widetilde{B}_j$ by Lemma~\ref{lem:prop}(ii).
	Thus, we have $\Lambda^2_\F(V)=A\oplus\bigoplus_{j=1}^{\lfloor \frac{f}{2}\rfloor}B_j$ by Lemma~\ref{lem:prop}(iii).

	Lemma~\ref{lem:prop}(i) deduces that
	\[(\Lambda_\E^2(U)_\F)^\E\cong \bigoplus_{i=1}^f\Lambda_\E^2(U)^{\phi^i}=\widetilde{A}\cong A^\E.\]
	Thus, $A\cong \Lambda_\E^2(U)_\F$ by Lemma~\ref{lem:prop}(iii), as in~(1) of part~(ii).
	Similarly, we obtain that $B_j\cong \bigl(U\otimes U^{\phi^j}\bigr)_\F$, when $1\leqs j < f/2$, as in~(2) of part~(ii).
	When $j$ is even, we have
	\[\left(\bigl(U\otimes U^{\phi^{f/2}}\bigr)_\F\right)^\E\cong \bigoplus_{i=1}^f\bigl(U\otimes U^{\phi^{f/2}}\bigr)^{\phi^i}\cong 2\cdot \bigoplus_{i=1}^{f/2}\bigl(U\otimes U^{\phi^{f/2}}\bigr)^{\phi^i}\cong 2\cdot \widetilde{B}_{f/2}.\]
	Hence, $2\cdot B_{f/2}\cong U\otimes U^{\phi^{f/2}}$ deduced by Lemma~\ref{lem:prop}(iii), as in~(3) of part~(ii).
\end{proof}
\begin{remark}
	Let $V$ be an irreducible $\F G$-module.
	Then $\E:=\End_{\F G}(V)$ is an extension field of $\F$ for  $V$ by Schur's Lemma and Wedderburn's Little Theorem.
	Then $V$ has an $\E G$-module structure defined by $\alpha g:v\mapsto v^{\alpha g}=v^{g\alpha}$ for $\alpha\in\E$ and $g\in G$.
	Let $U$ be the $\E G$-module of $V$ defined as above.
	Then $U_\F\cong V$.
	Hence, the decomposition of exterior square of an irreducible $\F G$-modules can be well-treated by extending the field to $\End_{\F G}(V)$.
\end{remark}

Let $V=\bbF_2^n$ be a vector space, and let $G\leqs\GL(V)\cong\GL_n(2)$ be transitive on $V\setminus\{0\}$.
These groups $G$ are known by the classification of the finite $2$-transitive permutation groups, and we list them in the next theorem, see~\cite[Table~7.3]{cameron1999Permutation}.

\begin{theorem}\label{thm:2trans}
	Suppose that $G\leqs\GL_n(2)$ acts transitively on non-zero vectors of $V=\bbF_2^n$.
	Then one of the following statements holds:
	\begin{enumerate}
		\item $G\leqs\GammaL_1(2^n)$ is solvable;
		\item $(G,n)=(\A_6,4)$, $(\Sp_{4}(2),4)$, $(\A_7,4)$, $(\PSU_3(3),6)$ or $(\G_2(2),6)$;
		\item $G^{(\infty)}$ is transitive on non-zero vectors and is one of the following groups:
		\begin{enumerate}
			\item $\SL_m(2^f)$ with $m\geqs 2$ and $(m,f)\neq(2,1)$;
			\item $\Sp_{2m}(2^f)$ with $m\geqs 2$ and $(m,f)\neq (2,1)$; or
			\item $\G_2(2^f)$ with $f\geqs 2$.
		\end{enumerate}
	\end{enumerate}
\end{theorem}

The following lemma can be directly deduced by Steinberg's Twisted Tensor Product Theorem~\cite{steinberg1963Representations} (also see~\cite[Theorem 5.4.1]{kleidman1990subgroup}).
\begin{lemma}\label{lem:stein}
	Suppose that $G\cong \Sp_{2m}(p^f)$, $\SL_m(p^f)$ or $\G_2(p^f)$ acting naturally on $\bbF_{p^f}$-space $U$ for prime $p$.
	Let $\phi:x\mapsto x^p$ be the Frobenius automorphism of $\bbF_{p^f}$.
	For $1\leqs i,j,k,\ell\leqs f-1$, the following statements hold:
	\begin{enumerate}
		\item $U^{\phi^i}\otimes U^{\phi^j}$ is an irreducible $\bbF_{p^f}G$-module for $i\neq j$;
		\item for $i\neq j$ and $k\neq \ell$,
		$U^{\phi^i}\otimes U^{\phi^j}\cong U^{\phi^k}\otimes U^{\phi^\ell}$ if and only if $\{i,j\}=\{k,\ell\}$.
	\end{enumerate}
\end{lemma}

We say an $\E G$-module $U$ can be \textit{written over} a subfield $\F\leqs\E$ if the matrix of each $g\in G$, with respect to some basis, has entries in $\F$.
Hence, $U$ can be written over $\F$ if and only if $U^\sigma\cong U$ for any $\sigma\in\Gal(\E/\F)$.
In addition, $U_\K$ is irreducible for any subfield $\F$ of $\E$ if and only if $U$ cannot be written over any subfield, see~\cite[26.5]{aschbacher2000Finite}.
Immediately we obtain the following lemma.
\begin{lemma}\label{lem:Bj}
	Suppose that $G\cong \Sp_{2m}(q)$, $\SL_m(q)$ or $\G_2(q)$ with $q=p^f$ for some prime $p$.
	Let $V$ be the natural $\bbF_pG$-module, and let $U$ be the natural $\bbF_qG$-module. Then
	\begin{enumerate}
		\item $V^{\bbF_q}\cong U\oplus U^\phi\oplus\cdots\oplus U^{\phi^{f-1}}$, where $\phi:x\mapsto x^p$ the Frobenius automorphism of $\bbF_q$;
		\item $(U\otimes U^{\phi^j})_{\bbF_p}$ is irreducible for $j<\frac{f}{2}$;
		\item when $f$ is even, $(U\otimes U^{\phi^{f/2}})_{\bbF_p}$ is a direct sum of two isomorphic copies of irreducible $\bbF_p G$-modules.
	\end{enumerate}
\end{lemma}
\begin{proof}
	For convenience, let $\E=\bbF_{q}$, $\F=\bbF_p$, and let $\L=\bbF_{p^{f/2}}$ when $f$ is even.

	It is clear that $U_\F\cong V$, and then part~(i) holds by Lemma~\ref{lem:decompoext}(i).

	Suppose that $1\leqs j< f/2$.
	Then we have $(U\otimes U^{\phi^j})^{\phi^k}\cong U\otimes U^{\phi^i}$ if and only if $f\mid k$ by lemma~\ref{lem:stein}.
	Hence, $U\otimes U^{\phi^j}$ cannot be written over any subfield of $\E$, and thus $(U\otimes U^{\phi^j})_\F$ is irreducible, as in part~(ii).

	Assume from now that $f$ is even.
	Note that $\Gal(\E/\L)=\langle\phi^{f/2}\rangle$ and $U\otimes U^{\phi^{f/2}}$ is normalized by $\phi^{f/2}$.
	By Lemma~\ref{lem:prop}(ii), there exists an $\L G$-module $W$ such that $W^{\E}\cong U\otimes U^{\phi^{f/2}}$.
	Lemma~\ref{lem:prop}(i) deduces that
	\[((U\otimes U^{\phi^{f/2}})_\L)^{\E}\cong (U\otimes U^{\phi^{f/2}})\oplus (U\otimes U^{\phi^{f/2}})^{\phi^{f/2}}\cong 2\cdot U\otimes U^{\phi^{f/2}}\cong 2\cdot W^{\E}.\]
	Hence, $(U\otimes U^{\phi^{f/2}})_\L=2\cdot W$ by Lemma~\ref{lem:prop}(iii).
	In particular, we have
	\[(U\otimes U^{\phi^{f/2}})_{\F}=((U\otimes U^{\phi^{f/2}})_{\L})_{\F}=2\cdot W_\F.\]
	We only need to show that $W_\F$ is an irreducible $\F G$-module, that is, $W$ cannot be written over any subfield of $\L$.

	Let $\psi$ be the image of $\phi$ in $\Gal(\L/\F)$.
	Then $\Gal(\L/\F)=\langle\psi\rangle$.
	Suppose that $W^{\psi^k}\cong W$ for some $k$.
	We have 
	\[\begin{aligned}
		W^{\psi^k}\cong W&\Rightarrow (2\cdot W)^{\psi^k}\cong 2\cdot W\Rightarrow((U\otimes U^{\phi^{f/2}})_\L)^{\psi^k}\cong (U\otimes U^{\phi^{f/2}})_\L\\
		&\Rightarrow((U\otimes U^{\phi^{f/2}})^{\phi^k})_\L\cong (U\otimes U^{\phi^{f/2}})_\L\\
		&\Rightarrow((U^{\phi^k}\otimes U^{\phi^{f/2+k}})_\L)^{\E}\cong ((U\otimes U^{\phi^{f/2}})_\L)^{\E}\\
		&\Rightarrow 2\cdot U^{\phi^k}\otimes U^{\phi^{f/2+k}}\cong 2\cdot U\otimes U^{\phi^{f/2}}.
	\end{aligned}\]
	Note that $U^{\phi^k}\otimes U^{\phi^{f/2+k}}\cong U\otimes U^{\phi^{f/2}}$ if and only if $k$ is divisible by $\frac{f}{2}$ by Lemma~\ref{lem:stein}.
	Hence, we have $\psi^k=1$, and then $W$ cannot be written over any subfield of $\L$.
	It follows that $W_{\F}$ is an irreducible $\F G$-module, which proves part~(iii).	
\end{proof}

We record some classical results for exterior squares of natural modules for $\SL_m(q)$, $\Sp_{2m}(q)$ and $\G_2(q)$.

\begin{lemma}[{\cite[Theorem 2.2]{liebeck1987affine}}]\label{lem:slext}
	Let $U=\bbF_q^m$, and let $G=\SL_m(q)$.
	Then $\Lambda^2_{\bbF_q}(U)$ is an irreducible $\bbF_q G$-module which cannot be written over any subfield of $\bbF_q$.

	In particular, if $m=3$, then $\Lambda^2_{\bbF_q}(U)$ is isomorphic to the dual $\bbF_q G$-module of $V$.
\end{lemma}

Results of exterior square of the natural modules of symplectic groups can be found in~\cite{carter2005Lie,korhonen2019Jordan,mcninch1998Dimensional}, and the result of the exceptional group of Lie type $\G_2(q)<\Sp_6(q)$ (even $q$) can be deduced from~\cite[Theorem 8.14]{seitz1987maximal}.

\begin{lemma}\label{lem:lambdasp}
	Suppose that $U=\bbF_q^{2m}$ is a symplectic space with $q=2^f$, and $G=\Sp_{2m}(q)$ or $\G_2(q)$($m=3$).
	Then $\Lambda^2_{\bbF_q}(U)$ has a maximal $\bbF_qG$-submodule $T$ of codimension $1$.
	Moreover, $T$ is indecomposable and one the following statements holds:
	\begin{enumerate}
		\item if $p\nmid m$ then $T$ is irreducible and cannot be written over any subfield;
		\item if $p\mid m$, then $\Lambda^2_{\bbF_q}(U)$ is indecomposable and $T$ has a unique maximal $\bbF_qG$-submodule $T_0$ which is trivial, and $T/T_0$ cannot be written over any subfield.
	\end{enumerate}
\end{lemma}

Now we are ready to prove Theorem~\ref{thm:dual}.

\begin{proof}[Proof of Theorem~$\ref{thm:dual}$]

	Let $G\leqs\GL(V)$ with $V=\bbF_2^n$, and let $M=\Lambda^2_{\bbF_2}(V)/W$ for a submodule $W$.
	Assume that $G$ is transitive on $V\setminus\{0\}$.
	Then $G$ is listed in Theorem~\ref{thm:2trans}.
	If $G$ is one of the small groups listed in Theorem~\ref{thm:2trans}(ii), then computations in \Magma~\cite{magma} shows that $\dim M=1$, which is a contradiction as $\dim V=\dim M=1$ implies $G=\GL_1(2)=1$.
	Thus, we only need to focus on the cases that $G^{(\infty)}\cong\SL_{m}(q)$ ($m\geqs 2$ and $(m,f)\neq(2,1)$), $\Sp_{2m}(q)$ ($m\geqs 2$ and $(m,q)\neq(2,2)$) or $\G_2(q)$ ($q\geqs 4$) for $q=2^f$.

	Let $K$ be the kernel of $G$ acting on $M$.
	Then $G/K$ is faithful and transitive on $M\setminus\{0\}$, and either $G^{(\infty)}\leqs K$ or $K\cap G^{(\infty)}\leqs\bfZ(G^{(\infty)})$.
	Suppose that $G^{(\infty)}\leqs K$.
	Then $G/K$ is isomorphic to a quotient group of $G/G^{(\infty)}\lesssim\GammaL(1,2^f)$.
	So
	\[|G/K|\leqs |\GammaL(1,2^f)|=f(2^f-1)\leqs \frac{n}{2}(2^{\frac{n}{2}}-1)<2^n-1=|M|-1,\]
	which is a contradiction since $G/K$ is intransitive on $M\setminus\{0\}$.

	We thus have that $K\cap G^{(\infty)}\leqs\bfZ(G^{(\infty)})$.
	So $G/K$ is faithful and transitive on $M\setminus\{0\}$, and has a simple composition factor isomorphic to one of $\PSL_m(q)$ ($m\geqs 2$ and $(m,f)\neq(2,1)$), $\PSp_{2m}(q)$ ($m\geqs 2$ and $(m,q)\neq(2,2)$) and $\G_2(q)$ ($q\geqs 4$) for even $q$.
	Note that none of these groups is isomorphic to $\A_6$, $\A_7$ or $\PSU_3(3)$, inspecting the candidates for $G/K$ in Theorem~\ref{thm:2trans}, we conclude that $(G/K)^{(\infty)}\cong G^{(\infty)}$.
	It follows that $K\cap G^{(\infty)}=1$ and $G^{(\infty)}$ is transitive on $M\setminus\{0\}$, and further $G^{(\infty)}$ is also transitive on $V\setminus\{0\}$.

	Let $U$ be the natural $\bbF_qG^{(\infty)}$-module, and let $\phi:x\mapsto x^2$ be the Frobenius automorphism of $\bbF_q$.
	Clear that $U_{\bbF_2}\cong V$.
	By Lemma~\ref{lem:decompoext}, $\Lambda^2_{\bbF_2}(V)$ is a direct sum of $\bbF_2G^{(\infty)}$-submodules $A$ and $B_j$'s for $1\leqs j\leqs\lfloor\frac{f}{2}\rfloor$, where	
	\begin{enumerate}
		\item $A\cong \Lambda_{\bbF_q}^2(U)_{\bbF_2}$;
		\item $B_j\cong (U\otimes U^{\phi^j})_{\bbF_2}$ for $1\leqs j<\frac{f}{2}$;
		\item when $f$ is even, $2\cdot B_{f/2}\cong (U\otimes U^{\phi^{f/2}})_{\bbF_2}$.
	\end{enumerate}
	Lemma~\ref{lem:Bj} deduces that each $B_j$ is an irreducible $\bbF_2G^{(\infty)}$-module.
	Suppose that $B_j$ is not a subspace of $W$ for some $j$.
	Then $M\cong B_j$ as $\bbF_2G^{(\infty)}$-modules.
	It follows that $\dim M=\dim B_j$.
	Note that $\dim B_j=n^2/f$ for $j\neq f/2$ and $\dim B_{f/2}=n^2/(2f)$ when $f$ is even.
	Hence, $M\cong B_j$ only when $j=f/2$, $f=n/2$ and $G^{(\infty)}\cong\SL_2(2^{f})$.
	In this case, the induced linear group of $G^{(\infty)}$ on $B_{f/2}$ is isomorphic to $\Omega^{-}_4(2^{f/2})$ (see \cite[page 45]{kleidman1990subgroup}), which cannot be transitive on non-zero vectors of $\bbF_2^{2f}$.
	Thus, $B_j\leqs W$ for each $j$, and then $M$ is $\bbF_2G^{(\infty)}$ isomorphic to a quotient of $A\cong(\Lambda^2_{\bbF_q}(U))_{\bbF_2}$.
	
	Suppose that $G^{(\infty)}\cong \Sp_{2m}(q)$ for $m\geqs 2$ (or $G^{(\infty)}\cong G_2(q)$) with $q=2^f$.
	Let $T$ and $T_0$ be $\bbF_q G^{(\infty)}$-submodules of $\Lambda^2_{\bbF_q}(U)$ give in Lemma~\ref{lem:lambdasp}.
	Assume that $m$ is even.
	Then $\Lambda^2_{\bbF_{q}}(U)$ is indecomposable with $T$ a maximal $\bbF_q G^{(\infty)}$-submodule by Lemma~\ref{lem:lambdasp}.
	Note that $\Lambda^2_{\bbF_{q}}(U)/T$ is a trivial $\bbF_{q}G^{(\infty)}$-module.
	This yields that $M$ is a trivial $\bbF_2 G^{(\infty)}$-module, a contradiction.
	Assume that $m$ is odd.
	Then $T$ is irreducible and cannot be written over any subfield by Lemma~\ref{lem:lambdasp}.
	Hence, $T_{\bbF_2}$ is an irreducible $\bbF_2 G^{(\infty)}$-module.
	This deduces that $M\cong S_{\bbF_2}$ as $\bbF_2T$-modules.
	Note that $\dim T_{\bbF_2}=f\cdot(\binom{n/f}{2}-1)$.
	By some elementary calculations, we obtain that $\dim T_{\bbF_2}=\dim M=n$ has no no integer solutions.

	Suppose that $G^{(\infty)}\cong \SL_m(q)$ for some $m\geqs 2$ and $q=2^f$.
	Lemma~\ref{lem:slext} says $\Lambda^2_{\bbF_q}(U)$ is an irreducible $\bbF_q G^{(\infty)}$-module and cannot be written over any subfield.
	Hence, $(\Lambda^2_{\bbF_q}(U))_{\bbF_2}$ is an irreducible $\bbF_q G_0$-module and $M\cong (\Lambda^2_{\bbF_q}(U))_{\bbF_2}$.
	Note that $\dim (\Lambda^2_{\bbF_q}(U))_{\bbF_2}=f\cdot \binom{n/f}{2}$ is equal to $\dim M=n$.
	We obtain that $f=n/3$, $m=3$ and $G^{(\infty)}\cong\SL_3(2^{n/3})$.
	In this case, $M\cong A\cong\Lambda^2_{\bbF_q}(U)_{\bbF_2}$ is isomorphic to the dual of $V$ deduced by Lemma~\ref{lem:slext}.
\end{proof}

\section{Proof of Theorem~\ref{thm:gross}}\label{sec:gross}

For a non-abelian special $p$-group $N$, the automorphism group $A=\Aut(N)$ naturally acts on the characteristic subgroup $\bfZ(N)$ and induces an action on the factor group $N/\bfZ(N)$.
Viewing the elementary abelian $p$-groups $N/\bfZ(N)$ and $\bfZ(N)$ as vector spaces over the field $\bbF_p$,
the actions of $A$ on them can be viewed as linear actions, and these give rise to two natural $\bbF_pA$-modules, denoted by $V$ and $M$, respectively.
The following lemma enables us to employ methods from group representation theory to study special $p$-groups.

\begin{lemma}\label{lem:spext}
	Let $N$ be a non-abelian special $p$-group with $p$ prime such that $N/\bfZ(N)\cong \bbZ_p^n$.
	Let $K$ and $L$ be the kernels of $A:=\Aut(N)$ acting on $V:=N/\bfZ(N)$ and $M:=\bfZ(N)$, respectively.
	Let $G=A/K$.
	Then the following statements hold:
	\begin{enumerate}
		\item{\rm\cite[Lemma 3]{glasby2011groups}} $K$ is an elementary abelian $p$-group of order $|\bfZ(N)|^n$;
		\item $K\leqs L$, and then both $V$ and $M$ are natural $\bbF_p G$-modules;
		\item the number of $A$-orbits on $N$ is equal to $o(V)+o(M)-1$ where $o(V)$ and $o(M)$ are numbers of $G$-orbits on $V$ and $M$ respectively;
		\item $M\cong\Lambda^2_{\bbF_p}(V)/W$ for some $\bbF_p G$-submodule $W$ of $\Lambda^2_{\bbF_p}(V)$.
	\end{enumerate}
\end{lemma}
\begin{proof}
	Let $C=\bfZ(N)=\Phi(N)=N'$ for convenience.
	By definition, $V:=N/C$ is a faithful $\bbF_pG$-module as $K$ is the kernel of $G$ on $V$.

	Suppose that $N=\langle x_1,...,x_n\rangle$.
	Then $M=N'$ is generated by commutators $[x_i,x_j]$ for $i,j=1,...,n$.
	By definition, each element $\sigma\in K$ fixes all of the cosets $\bfZ(N)x_i$, with $1\leqs i\leqs n$, so
	\[\mbox{$x_i^\sigma= a_ix_i$, where $a_i\in \bfZ(N)$ for each $i=1,...,n$.}\]
	Hence, $[x_i,x_j]^\sigma=[a_ix_i,a_jx_j]=[x_i,x_j]$ as $a_i,a_j\in \bfZ(N)$.
	Thus $K$ acts trivially on $M$, and so $K\leqs L$.
	It follows that $G=A/K$ acts on $M$ with kernel $L/K$, which yields that $M$ is an $\bbF_pG$-module, as in~(ii).

	Note that $A$ normalizes $C$, and partitions $N$ into two parts $C$ and $N\setminus C$.
	Thus, the number of $A$-orbits on $N$ is equal to the sum of numbers of $A$-orbits on $C$ and on $N\setminus C$.
	As usual, let `$\sim$' denote conjugation relation.

	First, as $K\le L$ and $L$ fixes every element of $C$, for any elements $a,b\in C=M$,
	\[\mbox{$a\sim b$ under $A$ $\Longleftrightarrow a\sim b$ under $A/L$ $\Longleftrightarrow a\sim b$ under $A/K=G$.}\]
	Thus the $A$-orbits on $C$ equal the $G$-orbits on $M=C$, and hence the number of $A$-orbits on $C$ is equal to $o(M)$.

	Next, as $K$ fixes each coset $Cx_i$, for each $\alpha\in K$, there exist $c_1,\dots,c_n\in C$ such that
	\[(x_1,\dots,x_n)^\alpha=(c_1x_1,\dots,c_nx_n).\]
	Thus $\alpha$ is uniquely determined by the tuple $(c_1,\dots,c_n)$, and so
	\[|K|\le|\{(c_1,\dots,c_n)\mid c_i\in C\}|=|C|^n.\]
	By Lemma~\ref{lem:spext}\,(i), we have $|K|=|C|^n$, and thus $|K|=|\{(c_1,\dots,c_n)\mid c_i\in C\}|$.
	Hence elements of $K$ and the tuples $(c_1,\dots,c_n)$ with $c_i\in C$ are one-to-one correspondent.
	It follows that elements of the same coset $Cx$ are conjugate under $K$.
	Hence, two elements $x,y\in N\setminus C$ are conjugate under $\Aut(N)$ if and only if the cosets $Cx,Cy$ are conjugate under $G=A/K$.
	Equivalently, the number of $A$-orbits on $N\setminus C$ is equal to the number of $G$-orbits on non-identity elements of $N/C=V$, which equals $o(V)-1$.
	Thus, the number of $A$-orbits on $N$ is equal to $o(V)+o(M)-1$, as in part~(iii).
	
	For $x,y,z\in N$, we notice that $[x,y]=[y,x]^{-1}$ and $[x,yz]=[x,y]^z[x,z]=[x,y][x,z]$ as $N'=\bfZ(N)$.
	Hence the commutator operator induces an alternating linear map:
	\[\mbox{$\varphi:\Lambda^2_{\bbF_p}(V)=V\wedge V\rightarrow M$ such that $\varphi: xC\wedge yC\mapsto [x,y]$.}\]
	For any element $g\in G$, we have that
	\[(xC\wedge yC)^{g\varphi}=[(xC)^g,(yC)^g]=[x,y]^g=(xC\wedge yC)^{\varphi g}.\]
	Thus $\varphi$ is an $\bbF_pG$-homomorphism from $\Lambda^2_{\bbF_p}(V)$ to $M$, and $M$ is isomorphic to $\Lambda^2_{\bbF_p}(V)/W$ for some $\bbF_p G$-submodule $W$ of $\Lambda^2_{\bbF_p}(V)$, as in (iv).
\end{proof}

Now we are ready to produce a list of finite $p$-groups with a single fusion class of elements of order $p$, and prove Theorem~\ref{thm:gross}.

\begin{proof}[Proof of Theorem~$\ref{thm:gross}$]

	Let $N$ be a $p$-group of which subgroups of order $p$ are all fused.

	If $N$ is abelian, then obviously $N$ is homocyclic.
	If $p$ is an odd prime, then $N$ is a homocyclic group by Shult~\cite{shult1969finite}.
	If $p=2$ and $N$ has only one involution, then $N$ is a generalized quaternion group.

	We thus assume that $N$ is a non-abelian $2$-group which has more than one involution, and $\Aut(N)$ is transitive on the involutions.
	By Gross \cite{gross1976automorphic}, one of the following holds:
		\begin{enumerate}
			\item $N$ has exponent $4$ and nilpotence class $2$, $|N|=|\bfZ(N)|^2$ or $|\bfZ(N)|^3$, and $N'=N^2=\bfZ(N)$;
			\item $N$ has exponent $8$ with nilpotence class $3$.
		\end{enumerate}
	By Bryukhanova~\cite{bryukhanova1981Automorphism}, a group $N$ in case~(i) with $|N|=|\bfZ(N)|^3$ indeed has $\Aut(N)$ solvable, and is a Suzuki $2$-group.
	Wilkens~\cite{wilkens1996note} excluded groups in case~(ii), that is, the nilpotence class of $N$ is not equal to $3$.

	To complete the proof, we may assume that $N$ is a non-abelian special $2$-group of order equal to $|\bfZ(N)|^2$.

	Then $\Phi(N)=N'N^2=N'=\bfZ(N)$, and $N$ is a special $2$-group.
	Let $V$ and $M$ be the vector space corresponding to $N/\bfZ(N)$ and $\bfZ(N)$ respectively.
	Let $G=\Aut(N)^{N/\bfZ(N)}$, the group induced by $\Aut(N)$ on the factor group $N/\Phi(N)$.
	Lemma~\ref{lem:spext} shows that $V$ and $M$ are $\bbF_2G$-modules for $G$, and $M$ is $\bbF_2 G$-isomorphic to a quotient of $\Lambda_{\bbF_2}^2(V)$.
	
	First, we show that $G$ acts transitively on non-zero vectors of both $V$ and $M$.
	Recall that $\bfZ(N)=N'$ consists of all involutions of $N$ and the identity.
	Hence, non-zero vectors of $M$ are in the same orbits of $\Aut(N)$.
	Since $G^{\bfZ(N)}=\Aut(N)^{\bfZ(N)}$ by Lemma~\ref{lem:spext}(ii), it follows that $G$ is transitive on non-zero vectors of $M$.
	Note that every element in $N\setminus \bfZ(N)$ is of order $4$.
	Suppose that $x,y\in N\setminus \bfZ(N)$.
	Then $x^2$ and $y^2$ are involutions, and hence there exists $\alpha\in \Aut(N)$ such that $(x^2)^\alpha=y^2$.
	This yields that $(x^\alpha)^2=y^2$.
	We may decompose $N=x_1\bfZ(N)\cup x_2\bfZ(N)\cup\cdots\cup x_{|\bfZ(N)|}\bfZ(N)$ for some $x_i\in N\setminus \bfZ(N)$ as $|N/\bfZ(N)|=|\bfZ(N)|$.
	Then, for each representative $x_i$ and any element $a\in M$, we have $(ax_i)^2=a^2x_i^2=x_i^2$.
	Thus, we have
	\[\begin{aligned}
		(x^\alpha)^2=y^2&\Longleftrightarrow \mbox{$x^\alpha$ and $y$ both lie in the same coset $x_i\bfZ(N)$ for some $x_i$,}\\
		&\Longleftrightarrow x^\alpha \bfZ(N)=y\bfZ(N)\Longleftrightarrow (x\bfZ(N))^\alpha=y\bfZ(N).
	\end{aligned}\]
	Hence, any two non-identity elements of $V:=N/\bfZ(N)$ are in the same orbits of $\Aut(N)$.
	Thus, $G$ is transitive on non-zero vectors of $V$.

	Suppose that $\Aut(N)$ is non-solvable.
	Then $G$ is non-solvable by Lemma~\ref{lem:spext}(i).
	Since $\dim M=\dim V$, we have that $G^{(\infty)}\cong\SL_3(2^f)$ and $M$ is $\bbF_2G^{(\infty)}$-isomorphic to the dual $V$ by Theorem~\ref{thm:dual}.
	Let $v=x\bfZ(N)$ be a non-zero vector of $V$ where $x\in N\setminus\bfZ(N)$.
	Then $x$ is of order $4$, and $w:=x^2$ is a non-zero vector in $M$.
	Note that an element of $G^{(\infty)}$ stabilizing $v$ must stabilizes $w$ since $(ax)^2=x^2$ for any $a\in\bfZ(N)$, and hence $G^{(\infty)}_v\le G^{(\infty)}_w$, which contradicts the fact that $M$ is the dual $\bbF_2G^{(\infty)}$-module of $V$ for $G^{(\infty)}\cong\SL_3(2^f)$.
	Thus, $\Aut(N)$ is solvable, and $N$ is a Suzuki $2$-group by Shaw's Theorem~\cite{shaw1970Sylow}.
	The solvability of $\Aut(N)$ is proved by Bryukhanova in~\cite[Theorem 1.2]{bryukhanova1981Automorphism}.
\end{proof}

\section{Finite $2$-groups with three fusion classes}

In 2020, Bors and Glasby posed a paper~\cite{bors2020Finitea} which classified finite $2$-groups with exactly three fusion classes.
In this section, we introduce the families of $2$-groups with three fusion classes, and then obtain a similar classification as in~\cite{bors2020Finitea} by using Theorem~\ref{thm:gross} and the main Theorem in~\cite{dornhoff1970imprimitive}.

The first family is $A_2(n,\theta)$, a Suzuki $2$-group of type A, see~\cite[page 82]{higman1963Suzuki}.

\begin{definition}\label{def:apn}
Let $\theta\in\Gal(\bbF_{2^n}/\bbF_{2})$ such that $|\theta|\neq 1$ is odd.
For $a,b\in\bbF_{2^n}$, define
\[A_2(n,\theta)=\l M(a,b,\theta) \mid a,b\in\bbF_{2^n}\r,\]
where $M(a,b,\theta)$ is a matrix of the form
\[M(a,b,\theta)=\begin{pmatrix}1&a&b\\0&1&a^\theta\\0&0&1\end{pmatrix}.\]
\end{definition}

We note that, if $n=2k+1$ and $x^\theta=x^{k+1}$, then $A_2(2k+1,\theta)$ is isomorphic to a Sylow $2$-subgroup of the Suzuki simple group $\Sz(2^{2k+1})$, see~\cite[Proposition 13.4(vi)]{carter1972Simple}.
For any odd integers $n\neq 1$, the groups $A_2(n,\theta)$ are AT-groups by Zhang in~\cite[Lemma~2.2]{zhang1992finite}.

The group $N=A_2(n,\theta)$ is a special $2$-group with $\bfZ(N)=\l M(0,b,\theta)\mid b\in \bbF_{2^n}\r\cong\bbZ_2^{n}$.
The matrix $D=\mathrm{diag}(\lambda^{-1},1,\lambda^\theta)$ induces an automorphism $\xi$ of $N$, that is
\[\xi:M(a,b,\theta)\mapsto D^{-1}M(a,b,\theta)D=M(a\lambda,b\lambda\lambda^\theta,\theta),\]
and $|\xi|=2^n-1$.
The Frobenius automorphism $\phi$ of $\bbF_2^{2^n}$ induces an automorphism of $A_2(n,\theta)$ of order $n$,
and $\l\xi,\phi\r=\GammaL_1(2^n)$.

\begin{lemma}\label{lem:antheta}
	The group $N=A_2(n,\theta)$ is an AT-group, and $\Aut(N)\cong 2^{n^2}{:}\GammaL_1(2^{n})$.
\end{lemma}
\begin{proof}
	Obviously, $M(a,b,\theta)$ and $M(a\lambda,b\lambda\lambda^\theta,\theta)$ do not belong to the same coset of $\bfZ(N)$.
	It follows that $\langle\xi\rangle$ acts on $N/\bfZ(N)$ faithfully, and hence $\langle \xi,\phi\rangle\cong\GammaL_1(2^n)$ acts on $N/\bfZ(N)$ faithfully.
	Let $K$ be the kernel of $\Aut(N)$ acting on $N/\bfZ(N)$.
	Lemma~\ref{lem:spext}(i) deduces that $K\cong \bbZ_{2}^{n^2}$, and thus
	\[\mbox{$\langle K,\xi,\phi\rangle\cong 2^{n^2}{:}\GammaL_1(2^n)\leqs\Aut(N)$.}\]
	Noting that $\Aut(N)$ is solvable by Theorem~\ref{thm:gross}(iii), we have that
	$\Aut(N)^{N/\bfZ(N)}$ is isomorphic to a solvable subgroup of $\GL_n(2)$ which is transitive on the non-zero vectors of $\bbF_2^n$.
	By Theorem~\ref{thm:2trans}, we obtain that $N/K=\Aut(N)^{N/\bfZ(N)}\lesssim \GammaL_1(2^n)$, and so
	$\Aut(N)\cong 2^{n^2}{:}\GammaL_1(2^{n})$.
\end{proof}

Sylow $2$-subgroups of $\SU_3(2^n)$ are Suzuki $2$-group of type B (see~\cite[page 314]{huppert1982Finite}).
Equip the space $\bbF_{q^2}^3$ ($q=2^n$) with the unitary form:
\[\bigl((x_1,x_2,x_3),(y_1,y_2,y_3)\bigr)=x_1y_3^q+x_2y_2^q+x_3y_1^q.\]
Then the following group $B_2(n)$ is a Sylow $2$-subgroup of $\SU_3(q)$, see~\cite[page 249]{dixon1996Permutation}.

\begin{definition}\label{def:bpn}
Let $q=2^n$, and $B_2(n)=\l M(a,b)\mid a,b\in\bbF_{q^2}, b+b^q+aa^q=0\r$, where
	\[M(a,b)= \begin{pmatrix}1&a&b\\0&1&a^q\\0&0&1\end{pmatrix}.\]
\end{definition}

It was noticed in Zhang~\cite[Lemma 2.3]{zhang1992finite} that $B_2(n)$ is an AT-group, with $\bfZ(N)=\l M(0,b)\mid b\in \bbF_{q^2}, b+b^q=0\r\cong\bbZ_2^{n}$.

\begin{lemma}\label{lem:bpn}
	Let $N=B_2(n)$.
	Then $N$ is an AT-group, and $\Aut(N)\cong 2^{2n^2}{:}\GammaL_1(2^{2n})$.
\end{lemma}
\begin{proof}
	Let $D=\mathrm{diag}(\lambda^{-q},\lambda^{1-q},\lambda)$, where $\l\lambda\r=\bbF_{q^2}^\times$.
	Then $D^{-1}M(a,b)D=M(a\lambda,b\lambda^{q+1})$, and
	$b\lambda^{q+1}+(b\lambda^{q+1})^q+(a\lambda)(a\lambda)^q=\lambda^{1+q}(b+b^q+a^{1+q})=0$.
	Thus, $D$ induces an automorphism of $N$, denoted by $\xi$.
	Let $\phi$ be the Frobenius automorphism $x\mapsto x^2$ on $\bbF_{q^2}$, which induces an automorphism of $N$.
	As $M(a,b)$ and $M(a\lambda,b\lambda^{q+1})$ are not in the same coset of $\bfZ(N)$, it follows that $\langle\xi,\phi\rangle\cong\GammaL_1(2^{2n})$ acts faithfully on $N/\bfZ(N)$.

	Arguing as in the proof of Lemma~\ref{lem:antheta}, the kernel $K$ of $\Aut(N)$ on $N/\bfZ(N)$ is isomorphic to $\bbZ_2^{n^2}$, and $\Aut(N)/K=\l \xi,\phi\r$.
	Therefore, we conclude that $\Aut(N)\cong 2^{2n^2}{:}\GammaL_1(2^{2n})$, and $N$ is an AT-group.
\end{proof}

The following Suzuki $2$-group was first given in case~(viii) of the main Theorem in~\cite{dornhoff1970imprimitive}.

\begin{definition}\label{def:pepsilon}
For a generator $\epsilon$ of $\bbF_{2^6}^\times$, define $P(\epsilon)$ to be the group consisting of elements in $\bbF_{2^6}\times\bbF_{2^3}$ with the opteration of multiplication defined by
\[(a,x)(b,w)=(a+b,x+y+ab^2\epsilon+a^8b^{16}\epsilon^8),\mbox{ for $(a,x),(b,w)\in \bbF_{2^6}\times\bbF_{2^3}$.}\]
\end{definition}

\vskip0.1in
The next lemma shows that the definition of $P(\epsilon)$ is independent of the choice of $\epsilon$.

\begin{lemma}\label{lem:uniquepe}
For any two generators $\epsilon$ and $\lambda$ of $\bbF_{2^6}^\times$, we have $P(\epsilon)\cong P(\lambda)$.
\end{lemma}
\begin{proof}
	Let $\epsilon$ be a generator of $\bbF_{2^6}^\times$.
	Let $\Tr: x\mapsto x+x^8$ be the trace map on $\bbF_{2^6}$ over $\bbF_{2^3}$.
	For convenience, let $f_{\epsilon}(a)=\Tr(a^3\epsilon)$ and let $g_{\epsilon}(a,b)=\Tr\bigl((a+b)ab\epsilon\bigr)$.
	Calculation shows
	\[(a,x)^2=(0,f_\epsilon(a)),\mbox{ and }[(a,x),(b,y)]=\bigl(0,g_\epsilon(a,b)\bigr)\mbox{, in $P(\epsilon)$}.\]

	Let $X=(x_1,...,x_6)$ be a basis of $\bbF_{2^6}^+$, and $Z=(z_1,z_2,z_3)$ be a basis of $\bbF_{2^3}^+$.
	Then the generating relations of the group $P(\epsilon)$ can be obtained by the coefficients of $f_\epsilon(x_i)$ and $g_\epsilon(x_i,x_j)$ with respect to the basis $Z$ of $\bbF_{2^3}^+$ for each $i,j=1,...,6$.

	First, we show that $P(\epsilon)\cong P(\epsilon^{3k+1})$ for any generator $\epsilon$ and any integer $k$.
	Fix $Z=(1,\epsilon^9,\epsilon^{18})$.
	Set $x_i=\epsilon^{i+k-1}$ and $y_i=\epsilon^{i-1}$ for $i=1,...,6$.
	Then both $X=(x_1,...,x_6)$ and $Y=(y_1,...,y_6)$ are basis of $\bbF_{2^6}^+$.
	Note that
	\[\begin{aligned}
		f_{\epsilon^{3k+1}}(a)&=\Tr(a^3\epsilon^{3k+1})=\Tr((a\epsilon^k)^3\epsilon)=f_{\epsilon}(a\epsilon^k),\\
		g_{\epsilon^{3k+1}}(a,b)&=\Tr\bigl((a+b)ab\epsilon^{3k+1}\bigr)=\Tr\bigl(((a\epsilon^k)+(b\epsilon^k))(a\epsilon^k)(b\epsilon^k)\bigr)=g_{\epsilon}(a\epsilon^k,b\epsilon^k).
	\end{aligned}\]
	Thus, under the basis $Z$, the coefficients of $f_{\epsilon}(x_i)$ and $f_{\epsilon^{3k+1}}(y_i)$ are equal, and the coefficients of $g_{\epsilon}(x_i,x_j)$ and $g_{\epsilon^{3k+1}}(y_i,y_j)$ are equal for $1\leqs i,j\leqs 6$.
	Hence $P(\epsilon)\cong P(\epsilon^{3k+1})$ for each integer $k$.

	We thus conclude that, for a given generator $\epsilon$, if $s\equiv t\pmod 3$ then $P(\epsilon^s)\cong P(\epsilon^t)$.
	Note that $\epsilon^{3k}$ is not an generator of $\bbF_{2^6}^\times$ for any integer $k$.
	Hence, for any generator $\lambda$ of $\bbF_{2^6}^\times$, $P(\lambda)$ is isomorphic to one of $P(\epsilon)$ and $P(\epsilon^2)$.

	Choose $X=(1,\epsilon,...,\epsilon^5)$ and $Z=(1,\epsilon^9,\epsilon^{18})$.
	It is not hard to see that the coefficients of $f_\epsilon(x_i)$ and $g_\epsilon(x_i,x_j)$ with respect basis $Z$ only depend on the minimal polynomial of $\epsilon$ over $\bbF_{2}$.
	Notice that the minimal polynomials of $\epsilon$ and $\epsilon^2$ are the same since $x\mapsto x^2$ is the Frobenius automorphism of $\bbF_{2^6}$.
	Thus we have $P(\epsilon)\cong P(\epsilon^2)$.

	Therefore, $P(\lambda)$ is isomorphic to $P(\epsilon)$ for any generators $\lambda$ and $\epsilon$ of $\bbF_{2^6}^\times$.
\end{proof}

The automorphism group of $P(\epsilon)$ can be computed by \Magma, led to the next lemma.

\begin{lemma}\label{lem:epsilon}
	Let $N=P(\epsilon)$ for a generator $\epsilon$ of $\bbF_{2^6}^\times$.
	Then $N$ is an AT-group with $\Aut(N)\cong 2^{18}{:}(7{:}9)$.
\end{lemma}

We remark that some generators $\epsilon$ of $\bbF_{2^6}^\times$ satisfy minimal polynomial $\epsilon^6+\epsilon^4+\epsilon^3+\epsilon+1$, and with this polynomial we may write the generating relations for $P(\epsilon)$ explicitly, which makes computation with $P(\epsilon)$ in Magma efficiently:
\[\begin{aligned}
	P(\epsilon)=&\langle x_i,z_j\mbox{ for $1\leqs i\leqs 6$ and $1\leqs j\leqs 3\r$, where}
\end{aligned}\]
\[\begin{array}{ll}
	& z_j^2=[x_i,z_j]=[z_k,z_{\ell}]=1\mbox{ for each $i,j,k,\ell$,}\\
	&x_1^2=x_3^2=z_2,\ x_2^2=z_2z_3,\ x_4^2=z_3,\ x_5^2=z_1z_2z_3,\ x_6^2=z_3,\\
	&[x_1,x_2]=[x_3,x_5]=[x_3,x_6]=z_1z_2,\ [x_1,x_3]=z_1z_3,\ [x_1,x_4]=z_3,\\
	&[x_1,x_5]=[x_3,x_4]=[x_5,x_6]=z_2,\ [x_1,x_6]=1,\ [x_2,x_6]=z_1z_2z_3,\\
	&[x_2,x_3]=[x_2,x_4]=[x_4,x_6]=z_1,\ [x_2,x_5]=[x_4,x_5]=z_2z_3\rangle.
\end{array}\]

Dornhoff~\cite{dornhoff1970imprimitive} determined finite groups $N$ of which a solvable group of automorphisms has exactly three orbits on $N$.
In his classification, desired non-abelian $2$-groups are precisely $A_2(n,\theta)$, $B_p(n)$ and $P(\epsilon)$.=
Then the following lemma immediately holds since each Suzuki $2$-group has a solvable automorphism group by Theorem~\ref{thm:gross}(iii).
\begin{lemma}\label{lem:at2gps}
	If $N$ is a finite non-abelian $2$-group such that $\Aut(N)$ has exactly three orbits on $N$, then $N=A_2(n,\theta)$, $B_p(n)$ or $P(\epsilon)$.
\end{lemma}

The above lemma is the same as~\cite[Theorem 1.1]{bors2020Finitea} given by Bors and Glasby.
They shows that $B_p(n)$ and $P(\epsilon)$ are in the family of Suzuki $2$-group of type B.

Now we are ready to prove Corollary~\ref{cor:AT-2-gps}.
Recall that two elements are conjugate under the automorphism group are said to be fused.

\begin{proof}[Proof of Corollary~$\ref{cor:AT-2-gps}$]
	First, let $N=A_2(n,\theta)$, $B_p(n)$ or $P(\epsilon)$, as in part~(iv) of Corollary~\ref{cor:AT-2-gps}.
	By Lemmas~\ref{lem:antheta}(ii), \ref{lem:bpn}(ii) and~\ref{lem:epsilon}, the group $N$ is an AT-group, as in part~(iii) of Corollary~\ref{cor:AT-2-gps}.
	This yields (iv)$\Rightarrow$(iii)$\Rightarrow$(ii)$\Rightarrow$(i), as by definition an AT-group is an FIF-group, and cyclic subgroups of an FIF-group $N$ of the same order are fused.

	Now assume that cyclic subgroups of $N$ of the same order are fused.
	Then involutions of $N$ are fused, and $N$ is a Suzuki $2$-group.
	Obviously, non-identity elements of $\bfZ(N)$ are all involutions, and elements in $N\setminus \bfZ(N)$ are all of order $4$.
	By definition, for any elements $x,y\in N\setminus\bfZ(N)$, there exists $\alpha\in \Aut(N)$ such that $x^\alpha=y$ or $x^\alpha=y^{-1}$.
	Hence, $(x\bfZ(N))^\alpha=y\bfZ(N)$, and $\Aut(N)$ is transitive on non-identity elements of $N/\bfZ(N)$.
	By Lemma~\ref{lem:spext}(iii), $\Aut(N)$ has exactly three orbits on $N$, and thus
	\begin{enumerate}
		\item[(a)] elements of the same order are in the same orbit of $\Aut(N)$, and $N$ is an AT-group, and also an FIF-group;
		\item[(b)] $N=A_2(n,\theta)$, $B_p(n)$ or $P(\epsilon)$ by Lemma~\ref{lem:at2gps}.
	\end{enumerate}
	That is to say, part~(i) of Corollary~\ref{cor:AT-2-gps} implies parts~(ii)-(iv).
	This deduces that parts~(i)-(iii) of Corollary~\ref{cor:AT-2-gps} are equivalent as obviously (iii)$\Rightarrow$(ii)$\Rightarrow$(i).
	So all the four statements of Corollary~\ref{cor:AT-2-gps} are equivalent, completing the proof.
\end{proof}

\end{document}